\documentclass[12pt]{article}
\usepackage[utf8]{inputenc}
\usepackage{scrextend}
\usepackage{amsmath}
\usepackage{graphicx}
\usepackage{amssymb}
\usepackage{amsthm}
\usepackage{hyperref}
\usepackage{cite}
\usepackage[left=3cm, right=3cm,top=3cm,bottom=3.5cm]{geometry}
\usepackage[lofdepth,lotdepth]{subfig}
\newtheorem{theorem}{Theorem}[section]
  \newtheorem{lemma}[theorem]{Lemma}

  \newtheorem{definition}[theorem]{Definition}

\title{Spanning trees without adjacent vertices of degree 2}
\author{Kasper Szabo Lyngsie\footnote{Department of Applied Mathematics and Computer Science, Technical University of Denmark, DK-2800 Lyngby, Denmark. E-mail addresses: ksly@dtu.dk, marmer@dtu.dk. The second author was supported by ERC Advanced Grant GRACOL, project number 320812, and by the Danish Council for Independent Research, Natural Sciences, grant DFF-8021-00249, AlgoGraph.}, $\text{Martin Merker}^*$}
\date{}

\begin{document}

\maketitle

\begin{abstract}
Albertson, Berman, Hutchinson, and Thomassen showed in 1990 that there exist highly connected graphs in which every spanning tree contains vertices of degree 2. Using a result of Alon and Wormald, we show that there exists a natural number $d$ such that every graph of minimum degree at least $d$ contains a spanning tree without adjacent vertices of degree 2. Moreover, we prove that every graph with minimum degree at least 3 has a spanning tree without three consecutive vertices of degree 2.
\end{abstract}

\section{Introduction}
All graphs in this paper are simple and finite unless stated otherwise. A \emph{homeo-morphically irreducible tree}, or simply a \emph{HIT}, is a tree without vertices of degree~2. HITs have been enumerated by Harary and Prins~\cite{harary} in 1959. A homeomorphically irreducible spanning tree of a graph is called a \emph{HIST}. The existence of HISTs in graphs with certain structures was studied by Hill~\cite{hill_1974} in 1974. Hill conjectured that any triangulation of the plane with at least 4 vertices contains a HIST. Malkevitch~\cite{malkevitch} made the even stronger conjecture that this also holds for all near-triangulations of the plane. Albertson, Berman, Hutchinson and Thomassen~\cite{Albertson} proved Malkevitch's conjecture in 1990 and asked the more general question whether every triangulation of a surface has a HIST. This was answered in the affirmative for the torus by Davidow, Hutchinson, and Huneke~\cite{hutchinson}, and for surfaces with sufficiently large representativity by Nakamoto and Tsuchiya~\cite{Naka}. A graph is \emph{locally connected}, if the neighbourhood of every vertex induces a connected subgraph. Chen, Ren, and Shan~\cite{chen_ren_shan_2012} proved the stronger result that any locally connected graph has a HIST, which implies that every triangulation of a surface has a HIST. Moreover, Chen and Shan~\cite{CHEN} proved that any graph where every edge is in two triangles contains a HIST, answering a question by Albertson et al.~\cite{Albertson}. Both proofs can also be found in the PhD thesis of Shan~\cite{Shan}. A graph is cyclically $k$-edge-connnected if the deletion of any $k-1$ edges does not result in a graph with two distinct components containing a cycle.
Hoffmann-Ostenhof, Noguchi, and Ozeki~\cite{Ostenhof} answered another question by Albertson et al. by showing that for every natural number~$k$ there exists a cubic graph which is cyclically $k$-edge-connected without HISTs.

A graph $G$ is called $H$-free if $G$ contains no induced subgraph isomorphic to~$H$. Furuya and Tsuchiya~\cite{FURUYA20132206} characterized the set of $P_4$-free graphs containing a HIST, where $P_k$ denotes the path on $k$ vertices. This characterization was extended to $P_5$-free graphs by Diemunsch et al.~\cite{DIEMUNSCH201571}. Albertson et al.~\cite{Albertson} showed that for some constant~$c$, any connected graph with $n$ vertices and minimum degree at least $c \sqrt{n}$ contains a HIST. In contrast to this, they also constructed a $k$-connected graph with no HIST for every natural number $k$. 

\begin{figure}
    \centering
    \includegraphics[scale = 0.6]{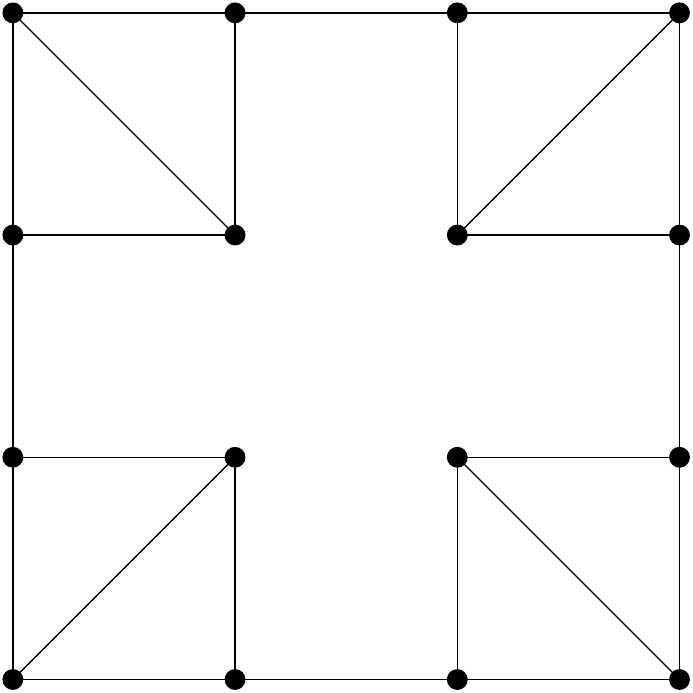}
    \caption{A cubic graph where every spanning tree has two adjacent vertices of degree~2.}
    \label{fig:counter_ex}
\end{figure}

In this paper we take a different approach and study a natural relaxation of the notion of homeomorphically irreducible spanning trees. We show that large constant minimum degree is sufficient for the existence of a spanning tree which is not far away from being homeomorphically irreducible. To be more precise, we construct spanning trees which can be obtained from HITs by subdividing each edge at most once. In other words, the vertices of degree 2 in the spanning tree form an independent set.

\begin{theorem} \label{thm:noadjdegree2}
There exists a natural number $d$ such that every connected graph with minimum degree at least $d$ has a spanning tree $T$ without adjacent vertices of degree 2.
\end{theorem}

Figure~\ref{fig:counter_ex} shows a cubic graph $G$ in which any spanning tree contains adjacent vertices of degree 2. This can be seen in the following way. Let $C_0, C_1, C_2, C_3$ be the four 4-cycles of $G$ and for $i \in \{1,2,3,4\}$ let $e_i$ be the edge joining  $C_{i}$ and $C_{i+1}$ (indices are taken modulo 4). Suppose $T$ is a spanning tree in $G$. We can assume that $e_0,e_1,e_2 \in E(T)$ and that $T$ restricted to $C_1 \cup \{e_1\} \cup C_2$ is connected. If $T$ has no adjacent vertices of degree 2, then it is easy to see that the restrictions of $T$ to $C_1$ and $C_2$ must be stars. The ends of $e_1$ now form two adjacent vertices of degree~2 in $T$, a contradiction.  Thus the number $d$ in Theorem~\ref{thm:noadjdegree2} has to be at least~4. Theorem~\ref{thm:min_degree_3_no_3_in_a_row} shows however, that minimum degree 3 is sufficient for the existence of a spanning tree with no path of length at least 2 consisting of vertices of degree~2.

\begin{theorem}\label{thm:min_degree_3_no_3_in_a_row}
Every connected graph with minimum degree at least 3 contains a spanning tree $T$ without three consecutive vertices of degree 2.
\end{theorem}

We give a proof of Theorem~\ref{thm:noadjdegree2} in Section 2. The proof of Theorem~\ref{thm:min_degree_3_no_3_in_a_row} is presented in Sections 3 and 4, where Section 3 contains the main part and Section 4 contains the proof of a technical lemma.

\section{Spanning trees without adjacent vertices of degree 2}

In this section we prove that if a graph $G$ has sufficiently large minimum degree, then $G$ has a spanning tree in which the vertices of degree 2 form an independent set. The main tool in our proof is a theorem on the existence of large star-factors in graphs of large minimum degree. A \emph{star} is a tree with at most one vertex of degree greater than~1. A \emph{star-cover} of a graph $G$ is a collection of vertex-disjoint stars in $G$ covering all vertices of $G$. The \emph{size} of a star, and more generally of a graph, is its number of edges. Answering a question of Havet et al.~\cite{Havet2011}, the following theorem was proved by Alon and Wormald~\cite{Alon}.

\begin{theorem} [Alon, Wormald~\cite{Alon}] \label{thm:bigdegreebigstar}
For every natural number $d$, there exists a natural number $f(d)$ such that every graph of minimum degree at least $f(d)$ has a star-cover where every star has size at least $d$.
\end{theorem}

Alon and Wormald showed that for every $\epsilon>0$ there exists a constant $c$ such that every graph of minimum degree at least $f(d)=cd^{3+\epsilon}$ has a star-cover where every star has at least $d$ edges. This result was improved by Nenadov~\cite{Nenadov} who showed that a minimum degree of $f(d)=c'd^2$ suffices, where $c'$ is a positive constant.

Given a star-cover $S$ of a graph $G$ in which every star has at least 3 edges, every spanning tree $T$ of $G$ containing $S$ does not contain three consecutive vertices of degree~2. An immediate consequence is a weaker version of Theorem~\ref{thm:min_degree_3_no_3_in_a_row} where the minimum degree 3 is replaced by a large constant. It seems plausible that by being more careful in the construction of $T$, one might obtain a spanning tree with even stronger properties. Alon and Wormald~\cite{Alon} asked in 2010 whether any graph with minimum degree $d$ has a spanning tree where all non-leaves have degree at least $cd/\log d$. However, Albertson et al.~\cite{Albertson} showed already in 1990 that large constant minimum degree does not even imply the existence of spanning trees without vertices of degree 2. Based on their examples, we construct for every natural number $k$ a series of graphs with arbitrarily large minimum degree where each spanning tree contains a vertex of degree~$i$ for every $i \in \{1,\ldots ,k\}$.

\begin{theorem}\label{thm:large_min_deg_small_degree_tree}
For all natural numbers $k,d$, there exists a graph $G(k,d)$ of minimum degree at least $d$ such that every spanning tree of $G(k,d)$ contains a vertex of degree $i$ for every $i\in \{1,\ldots ,k\}$.
\end{theorem}
\begin{proof}
The proof is by induction on $k$. Let a natural number $d$ be given. In the case $k=1$ we can choose $G(1,d)=K_{d+1}$, so we may assume $k \geq 2$. Let $G(k-1,d)$ be the graph given by the induction hypothesis and for each vertex $v \in V(G(k-1,d))$, let $K_v$ be a copy of $K_{d+1}$. Let $G(k,d)$ be obtained from the disjoint union of $G(k-1,d)$ and all the graphs $K_v$ for $v \in V(G(k-1,d))$ by adding an edge joining $v$ to one vertex in $K_v$ for each $v \in V(G(k-1,d))$. Clearly, $G(k,d)$ has minimum degree $d$. Let $T$ be a spanning tree of $G(k,d)$. Notice that the edges joining $v$ and $K_v$ are bridges and thus contained in $T$. Let $T'$ denote the subgraph of $T$ induced by $V(G(k-1,d))$. Clearly, $d_{T}(v) = d_{T'}(v)+1$ for every $v\in V(G(k-1, d))$ and $T'$ is a spanning tree of $G(k-1, d)$. By the induction hypothesis, $T'$ contains vertices of degree $1,\ldots,k-1$. Thus, $T$ has vertices of degree $2,\ldots,k$. Since $T$ is a tree, it also contains vertices of degree~1, so $G(k,d)$ is as desired.
\end{proof}

Theorem~\ref{thm:large_min_deg_small_degree_tree} shows that large minimum degree is not strong enough to avoid any specific degree in a spanning tree. Theorem~\ref{thm:noadjdegree2} shows that we can nevertheless obtain a spanning tree which is not too far away from being a HIST, in the sense that it can be obtained from a HIT by subdividing each edge at most once.

The strategy for proving Theorem \ref{thm:noadjdegree2} is to first apply Theorem~\ref{thm:bigdegreebigstar} to obtain a star-cover with stars of size at least 6. Given such a star-cover, we grow a tree~$T$ by starting with one of the stars and repeatedly adding stars together with some edges to $T$ (and possibly removing some edges), so that $T$ is always a tree without adjacent vertices of degree 2. To make sure that we do not create new vertices of degree 2 when we remove edges, we also require the vertices in $T$ which are adjacent to leaves with neighbours in $G-T$ to have degree at least~5 in $T$. We use this construction to prove the following lemma.

\begin{lemma} \label{lem:star}
Let $G$ be a connected triangle-free graph of minimum degree at least 3. If $G$ has a star-cover with stars of size at least 6, then $G$ has a spanning tree without adjacent vertices of degree 2.
\end{lemma}

\begin{proof}
Let $\{S_1,...,S_m\}$ be a star-cover of $G$ with stars of size at least 6 and let~$T$ be a tree in $G$ of maximal size with respect to the following conditions:
\begin{itemize}
    \item[(1)] $T$ contains no two consecutive vertices of degree 2, 
    \item[(2)] $V(G-T) = \bigcup_{i \in I}V(S_i)$ for some $I \subseteq \{1,...,m\}$, and
    \item[(3)] If $v \in V(T)$ is a leaf in $T$ that is joined to a vertex $u \in V(G-T)$ in $G$, then the neighbour of $v$ in $T$ has degree at least 5 in $T$. 
\end{itemize}
Clearly $T$ exists since $S_1$ satisfies the above properties. We can assume that $V(G-T)$ is non-empty since otherwise $T$ would be our desired spanning tree.\\
\par
\emph{\textbf{Claim 1: } If $v \in V(T)$ has a neighbour $u \in V(G-T)$ in $G$, then $v$ has degree 1 in $T$ and $u$ has degree 1 in $G-T$.}
\begin{proof}[Proof of Claim 1]
Let $v \in V(T)$ be adjacent to $u \in V(G-T)$ in $G$. By property (2), $u$ is contained in some star $S_j$ in $G-T$. If $v$ has degree greater than 1 in~$T$ or if $u$ has degree greater than 1 in $S_j$, then let $T'$ be the graph obtained from $T$ by adding the edge $uv$ and the star $S_j$. Clearly $T'$ is a tree satisfying (2) and since $u$ or $v$ had degree greater than 1, it contains no adjacent vertices of degree 2. The leaves of $T'$ are either leaves of $T$ or leaves of $S_j$ and since the centre of $S_j$ has degree at least 6 and all edges of $S_j$ are in $T'$, condition (3) is still satisfied. Thus, $T'$ satisfies all three conditions, contradicting our choice of $T$.\\ So we may assume that $v$ has degree $1$ in $T$ and $u$ has degree 1 in $S_j$. If $u$ has degree greater than 1 in $G-T$ but degree 1 in $S_j$, then, since $G$ is triangle-free, $u$ is adjacent to a vertex $w$ contained in some star $S_k$ in $G-T$ with $k\neq j$.
Now the graph obtained from $T$ by adding the stars $S_j, S_k$ and the edges $uv, uw$ contradicts the maximality of~$T$. 
\end{proof}
\par
\emph{\textbf{Claim 2: } In $G$ any vertex $v \in V(T)$ has at most one neighbour in $V(G-T)$.}
\begin{proof}[Proof of Claim 2]
Suppose $v \in V(T)$ has two neighbours $u,u' \in V(G-T)$. By property~(2) and Claim 1, $u$ and $u'$ are leaves in some stars $S_{j}, S_{j'}$ in $G-T$. If $S_{j} \neq S_{j'}$, then the tree obtained from the disjoint union of $T, S_{j}, S_{j'}$ by adding the edges $vu$ and $vu'$ contradicts the maximality of $T$. Hence $S_{j} = S_{j'}$. Let $w$ be the center of $S_j$. Let $T'$ be the tree obtained from the disjoint union of $T$ and $S_{j}$ by adding the edges $vu, vu'$ and removing the edge $u'w$. Clearly $T'$ satisfies conditions (1) and (2). Notice that the neighbour of the leaf $u'$ in $T'$ only has degree 3. Nevertheless, $T'$ also satisfies condition~(3) since $u'$ had degree 1 in $G-T$ and is thus not adjacent to any vertex in $G-T'$. Therefore $T'$ contradicts the maximality of $T$.
\end{proof}

Notice that Claim 2 implies that $T$ is not a star, since the minimum degree of $G$ is~3 and no two leaves of a star can be adjacent because $G$ is triangle-free.

Let $\mathcal{S}$ be the set of stars in $\{S_1,\ldots ,S_m\}$ which are disjoint from $T$ and have neighbours in $T$. Note that $\mathcal{S}$ is non-empty since $G$ is connected and $T$ is not a spanning tree. By Claim~1 and since $G$ has minimum degree 3, any star $S_i \in \mathcal{S}$ has a leaf $u_i$ which has at least two distinct neighbours $v_{i,1}, v_{i,2} \in V(T)$. For each star $S_i\in \mathcal{S}$ we now pick such vertices $u_i, v_{i,1}, v_{i,2}$. Let $G'$ be the graph obtained from $G$ by removing all the edges between $T$ and $G-T$ except the edges of type $v_{i,j}u_i$ where $j \in \{1,2\}$ and $S_i \in \mathcal{S}$. Furthermore, let $w_{i,j}$ denote the neighbour of $v_{i,j}$ in $T$ for $j\in \{1,2\}$, see Figure~\ref{fig:proof_star}.  

We now define an auxiliary graph $H$ which might have multiple edges and loops. The vertex set of $H$ is the subset of $V(T)$ consisting of all vertices $w_{i,j}$. Notice that it can happen that $w_{i,j} = w_{i',j'}$ for $(i,j)\neq (i',j')$, in which case the vertex $w_{i,j}$ is only included once in $V(H)$. The edge set of $H$ corresponds to the stars in $\mathcal{S}$: For each $S_i\in \mathcal{S}$, we have an edge $e_i$ between the vertices $w_{i,1}$ and $w_{i,2}$. We allow parallel edges, so $e_i\neq e_{i'}$ for $i\neq i'$. If $w_{i,1} = w_{i,2}$, then the edge $e_i$ is a loop at the vertex $w_{i,1}$.

We fix an almost balanced orientation of $H$, that is, an orientation such that the in-degree of a vertex of even degree equals the out-degree, and the in- and out-degrees of vertices of odd degree differ by 1. (To see that such an orientation exists, add a vertex joined to all vertices of odd degree and orient the edges consistently along an Euler walk.) We can assume, by relabelling the vertices $v_{i,j}$ and $w_{i,j}$ if needed, that every edge is oriented from $w_{i,1}$ to $w_{i,2}$. We construct a tree $T'$ from the disjoint union of $T$ and the stars in $\mathcal{S}$ by adding all the edges $u_iv_{i,2}$ and $u_iv_{i,1}$ and removing all the edges $w_{i,2}v_{i,2}$. Notice that since all vertices in $T$ have a non-leaf neighbour we have $d^-_H(v) \leq \lceil d_H(v)/2\rceil \leq \lceil (d_T(v)-1)/2\rceil$ for every $v \in V(H)$, where $d^-_H(v)$ denotes the in-degree of $v$ in $H$. Thus, 
$$d_{T'}(v) = d_T(v)-d^-_H(v) \geq d_T(v) - \left\lceil \frac{d_T(v)-1}{2}\right\rceil = \left\lceil \frac{d_T(v)}{2} \right\rceil \geq 3$$ 
for every $v\in V(H)$. Therefore, the set of vertices of degree 2 in $T'$ consists of the vertices of degree 2 in $T$ together with the vertices $v_{i,1}$. In particular, condition (1) is satisfied. Clearly $T'$ also satisfies condition (2). To see that~$T'$ also satisfies (3), note that $T'$ contains all stars in $\mathcal{S}$. Hence, the vertices of degree 1 in $T'$ with neighbours in $G-T'$ are contained in some stars in $\mathcal{S}$ and thus their neighbour in $T'$ has degree at least 6. This implies that $T'$ satisfies all three conditions and therefore contradicts the maximality of $T$. 
\end{proof}
\begin{figure}
    \centering
    \includegraphics{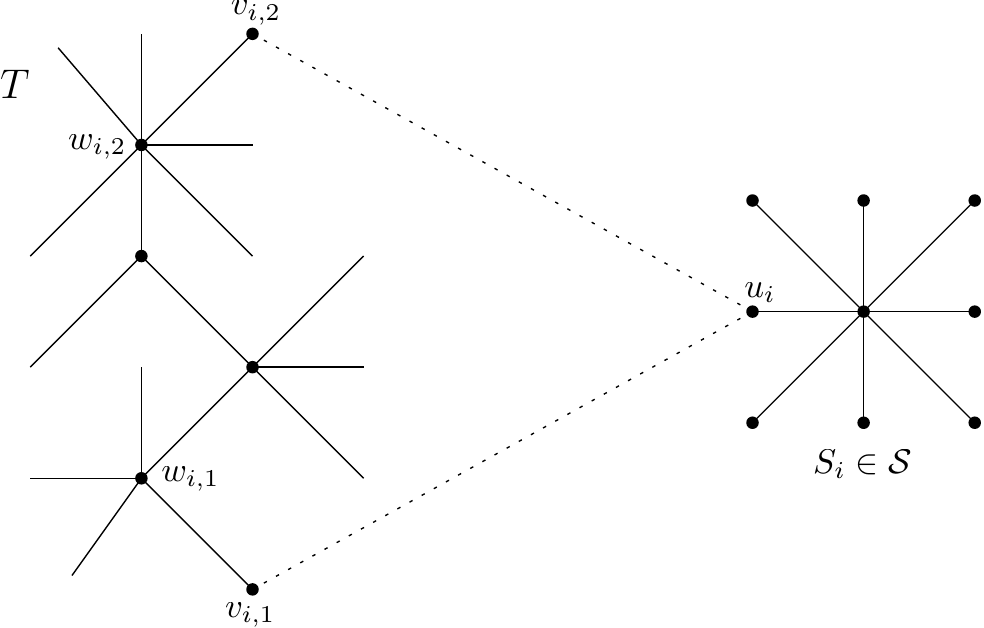}
    \caption{Definition of $u_i, v_{i,1}, v_{i,2}, w_{i,1}, w_{i,2}$ in Lemma \ref{lem:star}.}
    \label{fig:proof_star}
\end{figure}

Given Theorem \ref{thm:bigdegreebigstar} and Lemma \ref{lem:star}, we can easily prove Theorem~\ref{thm:noadjdegree2}.
\begin{proof}[Proof of Theorem \ref{thm:noadjdegree2}]
Let $G$ be a graph of minimum degree at least $2 f(6)$, where $f$ is the function defined by Theorem \ref{thm:bigdegreebigstar}. Let $H$ be a bipartite subgraph of $G$ of maximal size. Clearly $H$ is spanning and connected. Notice that $d_H(v)\geq \frac{1}{2}d_G(v)$ for every $v\in V(G)$ since otherwise we could move $v$ to the other bipartite class of $H$ to obtain a bipartite subgraph of $G$ of larger size. In particular, $H$ has minimum degree at least~$f(6)$. By Theorem \ref{thm:bigdegreebigstar}, $H$ has a star-cover where each star has size at least 6. Since $H$ is triangle-free, by Lemma \ref{lem:star}, the graph $H$ has a spanning tree without adjacent vertices of degree 2.
\end{proof}

\section{Spanning trees without 3 consecutive vertices of degree 2}

 The main theorem of this section is Theorem \ref{thm:no3inarow} below which immediately implies Theorem \ref{thm:min_degree_3_no_3_in_a_row}.

\begin{theorem} \label{thm:no3inarow}
Every simple connected graph $G$ has a spanning tree $T$, such that there is no path of length 2 in $T$ all of whose vertices have degree 2 in $T$ and degree at least 3 in $G$.
\end{theorem}

This stronger version of Theorem \ref{thm:min_degree_3_no_3_in_a_row} allows us to use induction on the size of~$G$. To simplify the notation in the following proofs, we introduce the following definition.

\begin{definition}
Let $H$ be a subgraph of $G$. We say a path of length 2 in $H$  is \emph{$G$-bad} if all its vertices have degree at least 3 in $G$ and degree 2 in $H$. We say a subgraph $H$ is \emph{$G$-bad} if it contains a $G$-bad path, otherwise we call it \emph{$G$-good}.
\end{definition}

Now the statement of Theorem~\ref{thm:no3inarow} is simply that every simple connected graph $G$ has a $G$-good spanning tree. First we show that to prove this statement, it is sufficient to consider graphs of minimum degree at least 3.

\begin{lemma} \label{lem:mincounter}
A minimal counterexample to Theorem \ref{thm:no3inarow} has minimum degree at least~3.
\end{lemma}
\begin{proof}
Let $G$ be a connected simple graph which has no $G$-good spanning tree and for which $|V(G)|$ is minimal. Clearly $|V(G)| \geq 4$.\\

\emph{\textbf{Claim 1: }$G$ has no vertices of degree 1.}

\begin{proof}[Proof of Claim 1]
Suppose $v \in V(G)$ has degree 1 and let $G'=G-v$. By minimality of $G$, we can find a spanning tree $T'$ in $G'$ which is $G'$-good. Let $T_1$ denote the graph we get by adding $v$ and its incident edge to $T'$. Clearly $T_1$ is a spanning tree of $G$. The only way how $T_1$ could be $G$-bad is that $v$ is adjacent to an endvertex $x$ of a $G$-bad path in $T_1$, say $xyz$. In this case, let $u$ denote a neighbour of $x$ different from $v$ and $y$. Now consider the graph $T_2=T_1-xy+xu$, which is another spanning tree of $G$. If $T_2$ is $G$-bad, then there must be a $G$-bad path $xuw$ in $T_2$, see Figure~\ref{fig:lem_3_3_claim_1}. In particular, the vertex $w$ has degree 2 in $T_2$. Finally, set $T_3=T_2-uw+xy$. It is easy to see that $T_3$ is a $G$-good spanning tree of $G$. \end{proof}
\par
\emph{\textbf{Claim 2: }$G$ has no vertices of degree 2.}

\begin{figure}[t]
    \centering
    \includegraphics[width=\textwidth]{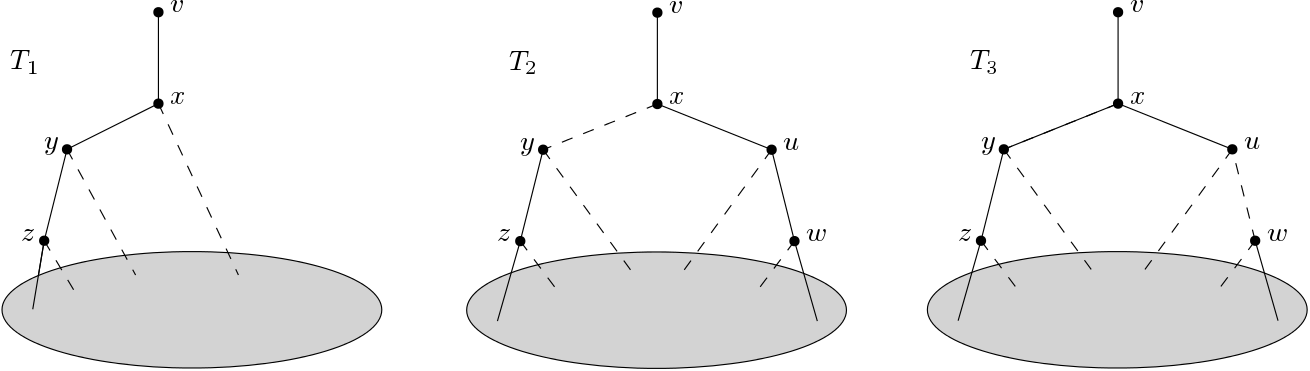}
    \caption{Proof of Claim 1 in Lemma \ref{lem:mincounter}}
    \label{fig:lem_3_3_claim_1}
\end{figure}

\begin{figure}[h]
    \centering
    \includegraphics[width=\textwidth]{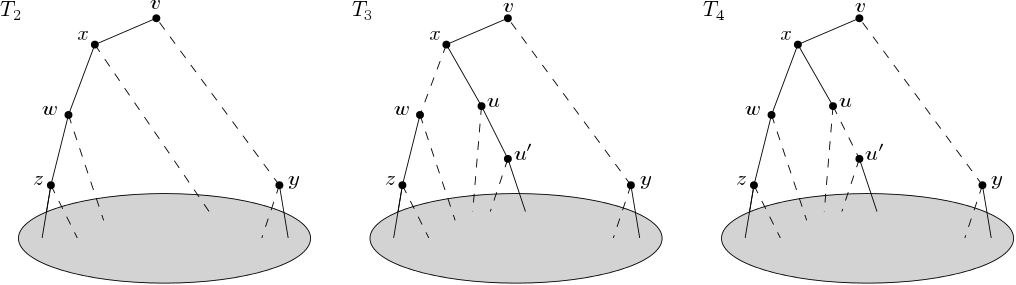}
    \caption{Proof of Claim 2 where $d(v)=2$ and $xy\notin E(G)$}
    \label{fig:lem_3_3_claim_2_1}
\end{figure}

\begin{proof}[Proof of Claim 2]
Suppose $v \in V(G)$ has exactly two neighbours $x, y$ and that $x$ and $y$ are non-adjacent. By minimality of $G$, the graph $G'=G-v+xy$ has a $G'$-good spanning tree $T'$. If $xy \in E(T')$, then $T_1=T'-xy+xv+yv$ is a $G$-good spanning tree, so we can assume that $xy \notin E(T')$. In this case we can assume that  $T_2=T'+xv$ is a $G$-bad spanning tree. Hence, $x$ is an endvertex in a $G$-bad path $xwz$ in $T_2$. Let $u$ be a neighbour of $x$ in $G$ different from $v$ and $w$. We can assume that $T_3=T_2-xw+xu$ contains a $G$-bad path $xuu'$. Notice that $u$ and $u'$ have degree 2 in $T_3$. Now $T_4=T_3-uu'+xw$ is a $G$-good spanning tree, see Figure~\ref{fig:lem_3_3_claim_2_1}. Thus we may assume that every vertex of degree~2 is contained in a triangle in $G$.

\begin{figure}
    \centering
    \includegraphics[width=\textwidth]{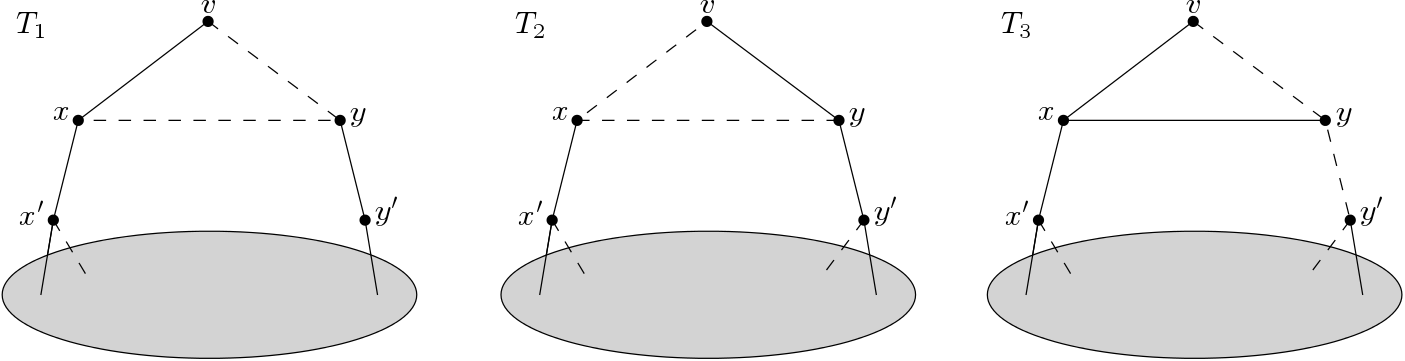}
    \caption{Proof of Claim 2 where $xy \in E(G)$ and $G'=G-v-xy$ is connected}
    \label{fig:lem_3_3_claim_2_2}
\end{figure}

If one of $x$ and $y$, say $x$, does not have degree 3 in $G$, then any $(G-vx)$-good spanning tree of $G-vx$ is also a $G$-good spanning tree, so by the minimality of $G$ we can assume that both $x$ and $y$ have degree 3 in $G$. Let $x'$ and $y'$ denote the neighbours of $x$ and $y$ which are different from $x$, $y$ and $v$. If $G'=G-v-xy$ is connected, then let $T'$ be a $G'$-good spanning tree of $G'$. If both  $T_1=T'+vx$ and $T_2=T'+vy$ are $G$-bad, then $x'$ and $y'$ have degree 2 in $T'$ and $T_3=T_1+xy-yy'$ is a $G$-good spanning tree, see Figure~\ref{fig:lem_3_3_claim_2_2}. Thus we may assume that $G'$ is disconnected. In particular $x' \neq y'$ and both $x'$ and $y'$ have degree at least 3 since $xx'$ and $yy'$ are not contained in triangles.

\begin{figure}
    \centering
    \includegraphics[scale=.45]{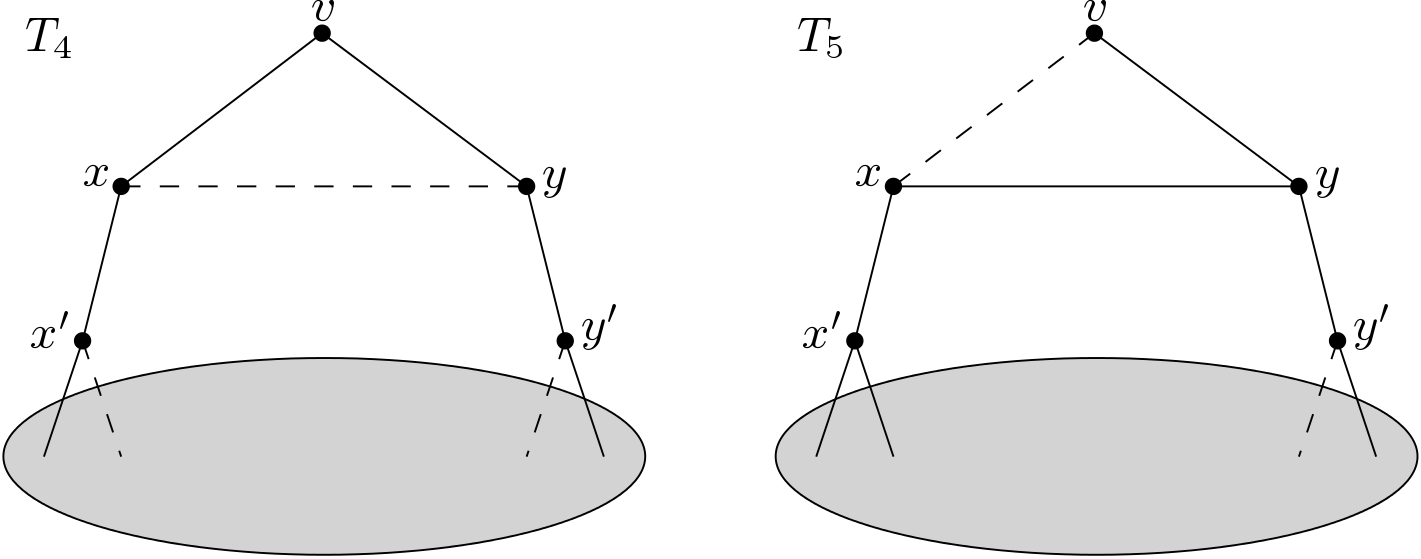}
    \caption{Proof of Claim 2 where $xy \in E(G)$ and $G'=G-v-xy$ is disconnected}
    \label{fig:lem_3_3_claim_2_3}
\end{figure}

Let $G''=G-v-x-y+x'y'$ and let $T''$ be a $G''$-good spanning tree of $G''$. Since~$G'$ is disconnected we have that $x'y' \in E(T'')$. If both $x', y'$ have degree 2 in $T''$, then $T_4=T''-x'y'+x'x+xv+vy+yy'$ is a $G$-good spanning tree of $G$. So one of $x',y'$ does not have degree 2 in $T''$, say $x'$. Now $T_5=T''-x'y'+x'x+xy+yv+yy'$ is a $G$-good spanning tree, see Figure~\ref{fig:lem_3_3_claim_2_3}.
\end{proof}
Claims 1 and 2 immediately imply that $G$ has minimum degree 3.
\end{proof}

Thomassen and Toft~\cite{Thomassen} proved that every connected graph $G$ with minimum degree~3 contains an induced cycle $C$ such that $G-V(C)$ is connected. Let $G$ be a minimal counterexample to Theorem~\ref{thm:no3inarow}. By Lemma~\ref{lem:mincounter}, $G$ has minimum degree at least 3, so there exists an induced cycle $C$ for which $G-V(C)$ is connected. In particular, also $G'= G-E(C)$ is connected. Note that if $C$ does not contain any vertices of degree 4 in $G$, then any $G'$-good spanning tree of $G'$ is also a $G$-good spanning tree of $G$, contradicting our choice of $G$. In particular, every non-separating induced cycle of $G$ contains a vertex of degree 4. This already shows that Theorem~\ref{thm:no3inarow} is true for subcubic graphs and Theorem~\ref{thm:min_degree_3_no_3_in_a_row} for cubic graphs. 

To prove Theorem~\ref{thm:no3inarow} in its full generality, we show that in a graph of minimum degree at least 3 we can always find an induced non-separating subgraph~$H$ with the property that we can extend every $(G-H)$-good spanning tree of $G-H$ to a $G$-good spanning tree of $G$. Two of the reducible structures we use are so-called $W_a$- and $W_{a,b}$-configurations which are defined as follows, see Figure \ref{fig:W_a-conf} and \ref{fig:W_a,b-conf}. 

\begin{definition}[$W_a$-configuration]
A $W_a$-configuration in $G$ is an induced subgraph $H$ consisting of a path $P=v_1\cdots v_a$ and three distinct vertices $v,x,y$ not contained in $P$, such that $v$ is adjacent to all vertices in $V(P) \cup \{x,y\}$, $xv_1, yv_a \in E(H)$, and every vertex in $V(H)\setminus \{v\}$ has degree 3 in $G$. Moreover, $G-H$ is connected, both $x$ and $y$ have precisely one neighbour in $G-H$ and no other vertex of $H$ has a neighbour in $G-H$. We call $v$ the \emph{centre} and $x,y$ the \emph{connectors} of the $W_a$-configuration.
\end{definition}

\begin{definition}[$W_{a,b}$-configuration]
A $W_{a,b}$-configuration in $G$ is an induced subgraph $H$ consisting of two disjoint paths $P=v_1\cdots v_a$, $Q=u_1\cdots u_b$, and three distinct vertices $v,x,y$ not contained in the paths such that $v$ is adjacent to all vertices in $V(P) \cup V(Q)$, $xv_1, xu_1, yv_a, yu_b \in E(H)$, and every vertex in $V(H)\setminus \{v\}$ has degree 3 in $G$. Moreover, $G-H$ is connected, both $x$ and $y$ have precisely one neighbour in $G-H$ and no other vertex of $H$ has a neighbour in $G-H$. We call $v$ the \emph{centre} and $x,y$ the \emph{connectors} of the $W_{a,b}$-configuration.
\end{definition}

\begin{figure}[h]
\centering
\subfloat[Subfigure 1 list of figures text][A $W_a$-configuration]{
\includegraphics[width=0.4\textwidth]{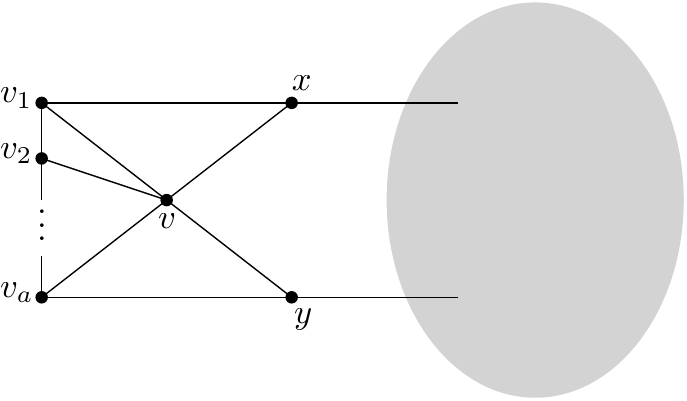}
\label{fig:W_a-conf}}
\qquad
\subfloat[Subfigure 2 list of figures text][A $W_{a,b}$-configuration]{
\includegraphics[width=0.4\textwidth]{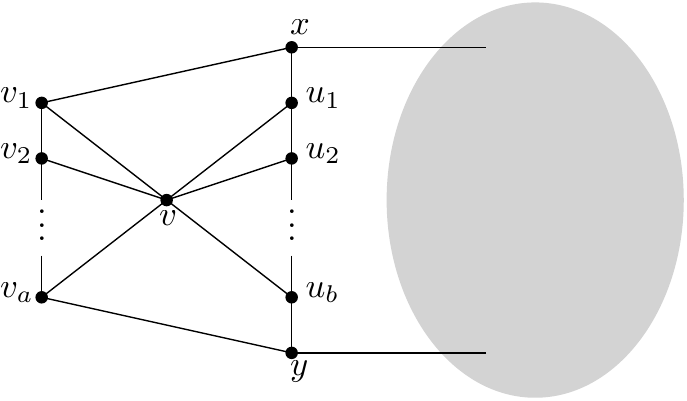}
\label{fig:W_a,b-conf}}
\caption{Two types of graph-configurations.}
\label{fig:graph_fig}
\end{figure}

Lemma~\ref{lem:main} allows us to find the reducible structures we need to finish the proof of Theorem~\ref{thm:no3inarow}.  The proof of this lemma is rather technical and therefore we postpone it to the next section.

\begin{lemma} \label{lem:main}
Let $G$ be a connected graph of minimum degree at least 3. Let $S$ be a set of vertices in $G$ containing all vertices of degree greater than 3 and possibly some vertices of degree 3. Then at least one of the following three conditions is satisfied:
\begin{description}
    \item[(C)] There exists an induced cycle $C$ containing no vertex of $S$ such that $G-E(C)$ is connected.
    \item[(P)] There exists an induced path $P$ with endvertices in $S$ such that $G-E(P)$ is connected.
    \item[(W)] There exists a $W_a$-configuration or a $W_{a,b}$-configuration in $G$ where the center is contained in $S$.
\end{description}
\end{lemma}

Notice that all three conditions are indeed necessary. To see that the statement is not true if we omit condition (W), we can consider the following construction. Let $T$ be any homeomorphically irreducible tree. Now for every leaf $t$ in $T$, we add a $W_a$- or $W_{a,b}$-configuration with both connectors joined to $t$. Let $G$ denote the resulting graph, and let $S$ denote the set of vertices which have degree at least 4 or are centres of the configurations. Now every non-trivial block in $G$ consists of a $W_a$- or $W_{a,b}$-configuration together with a vertex of degree 2. It is easy to see that every non-separating cycle in~$G$ contains precisely one vertex of $S$. Moreover, any path containing two vertices of~$S$ also contains a bridge and is therefore separating.

Finally, we use Lemma~\ref{lem:mincounter} and Lemma~\ref{lem:main} to finish the proof of Theorem~\ref{thm:no3inarow}.

\begin{proof}[Proof of Theorem~\ref{thm:no3inarow}]
Let $G$ be a minimal counterexample. By Lemma \ref{lem:mincounter} the minimum degree of $G$ is at least 3. Let $S \subset V(G)$ be the set of vertices in $G$ with degree at least $4$. By Lemma \ref{lem:main} it suffices to consider the following three cases. \\

\textbf{Case 1:} There exists a non-separating induced cycle $C$ containing no vertex of $S$.\\
By the minimality of $G$ there exists a spanning tree $T$ of $G'=G-E(C)$ which is $G'$-good. Now $T$ is also a $G$-good spanning tree of~$G$.   \\

\textbf{Case 2:} There exists an induced path $P$ with endvertices in $S$ for which $G-E(P)$ is connected. We may assume that no interior vertex of $P$ is contained in $S$ by considering a shortest such path. As in Case 1, by minimality of $G$ there exists a spanning tree $T$ of $G'=G-E(P)$ which is $G'$-good. Now $T$ is also a $G$-good spanning tree of $G$.    \\

\textbf{Case 3:} There exists a $W_a$-configuration or a $W_{a,b}$-configuration in $G$.\\
Let $H$ denote such a configuration with centre $v$ and connectors $x$ and $y$, and let $v_1$ denote a common neighbour of $x$ and $v$. By minimality of $G$, the graph $G'=G-(H-x-y)$ has a $G'$-good spanning tree $T$. We can obtain a $G$-good spanning tree of $G$ by adding all edges incident with $x$ and all edges incident with $v$ apart from $vv_1$ and $vy$.
\end{proof}

\section{Proof of Lemma~\ref{lem:main}}

Thomassen and Toft~\cite{Thomassen} proved several results about the existence of non-separating induced cycles. Note that a cycle is called \emph{non-separating} if $G-V(C)$ is connected. If $C$ is induced and all vertices of $C$ have degree 3, then $G-V(C)$ is connected if and only if $G-E(C)$ is connected. The result we use in this section states that under some mild conditions there exists an induced non-separating cycle or a $k$-rail avoiding some subgraph $G'$ of $G$. A \emph{k-rail} in a graph $G$ between two vertices $x$ and $y$ is a collection of $k$ internally disjoint paths joining $x$ and $y$ such that all interior vertices of the paths have degree 2 in $G$.

\begin{lemma}[Thomassen, Toft~\cite{Thomassen}]\label{Lemma-ThomassenToft}
 Let $G$ be a 2-connected graph and let $G'$ be a connected subgraph of $G$ such that $G-V(G')$ contains at least one cycle. Then $G-V(G')$ contains an induced cycle $C$ such that $G-V(C)$ is connected or $G-V(G')$ contains a $k$-rail $R$ for some $k\geq 3$ which is also a $k$-rail in $G$ such that $G-V(R)$ is connected.
\end{lemma}

Instead of Lemma~\ref{Lemma-ThomassenToft}, we use the following lemma which is an easy corollary and more suited for our purposes.

\begin{lemma} \label{lem:new_carsten_toft}
Let $G$ be a 2-connected graph and let $G'$ be a non-empty connected subgraph of $G$ such that $G'$ contains all vertices of degree at least 4 and $G-V(G')$ contains at least one cycle. Then $G-V(G')$ contains an induced cycle $C$ such that $G-V(C)$ is connected.
\end{lemma}
\begin{proof}
By Lemma \ref{Lemma-ThomassenToft} it suffices to show that $G-V(G')$ cannot contain a $k$-rail for $k \geq 3$. So suppose $R$ is such a $k$-rail between two vertices $x$ and $y$. Since $G'$ contains all vertices of degree at least 4 in $G$ and since $k \geq 3$, there can be no edges between $R$ and $G-R$, contradicting that $G$ is connected.
\end{proof}

Tutte showed that any pair of vertices in a 3-connected graph $G$ can be connected by an induced path $P$ such that $G-V(P)$ is connected. The following edge-version is an easy application of this theorem.  We give a short self-contained proof which we will also refer to in the proof of Lemma \ref{lem:main}.

\begin{lemma} \label{lem:tutte}
For any two vertices $v_1,v_2$ in a 3-edge-connected graph $G$, there exists an induced path $P$ from $v_1$ to $v_2$ such that $G-E(P)$ is connected. 
\end{lemma}

\begin{proof}
Let $P$ be a path from $x$ to $y$ which maximizes the size of the largest connected component of $G-E(P)$. Clearly, we may assume that $P$ is an induced path in $G$. Let~$K$ denote the largest component of $G-E(P)$. Notice that 
\begin{description}
    \item[(*)] for any vertices $z_1,z_2$ on $P$ belonging to the same component $L\neq K$ of $G-E(P)$, the $z_1z_2$-subpath of $P$ does not contain any vertices of $K$,
\end{description}
since otherwise we could replace this $z_1z_2$-subpath by a path in $L$ to obtain a new $v_1v_2$-path $P'$ for which the component of $G-E(P')$ containing $K$ is strictly larger than before, contradicting our choice of $P$. Let $k_1$ and $k_2$ denote the first and last vertex on $P$, respectively, which is contained in $K$. If $k_1\neq v_1$, then let $e$ denote the last edge of the $v_1k_1$-subpath of $P$. By (*), the edge $e$ is a cut-edge in $G$ which contradicts 3-edge-connectivity. 

Thus, we may assume that $k_1=v_1$ and similarly $k_2=v_2$. If $G-E(P)$ is not connected, then there exists a vertex on $P$ which is not in $K$. Let $w$ be the first such vertex on the path from $v_1$ to $v_2$. Let $k$ denote the first vertex on the $wv_2$-subpath of $P$ which is contained in $K$, see Figure~\ref{fig:proof_tutte}. Let $e_w$ and $e_k$ denote the last edge of the $v_1w$-subpath of $P$ and of the $v_1k$-subpath of $P$, respectively. By (*), the edges $e_w,e_k$ form a 2-edge-cut in $G$, contradicting 3-edge-connectivity. Thus, every vertex of $G-E(P)$ is contained in $K$ and $G-E(P)$ is connected.
\end{proof}

\begin{figure}
    \centering
    \includegraphics[scale = 0.9]{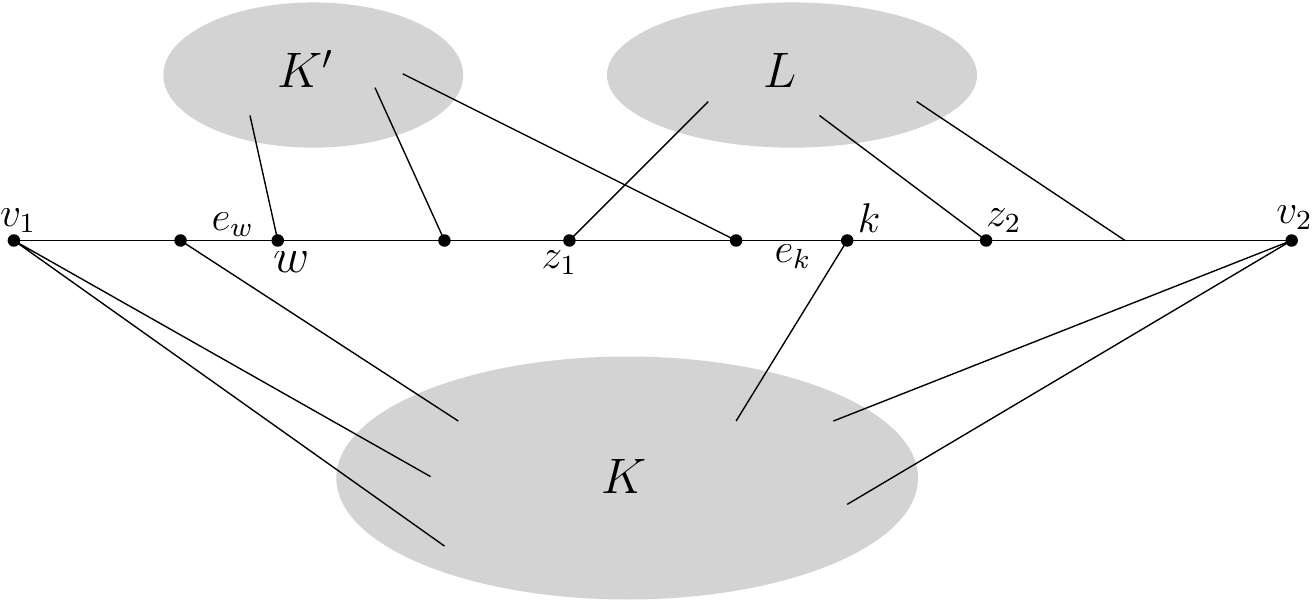}
    \caption{Proof of Lemma~\ref{lem:tutte}.}
    \label{fig:proof_tutte}
\end{figure}

In the proof of Lemma~\ref{lem:main} we investigate the block structure of $G$. We refer the reader to~\cite{Diestel} for the definitions of the block decomposition and block graph. An \emph{endblock} is a block which corresponds to a vertex of degree at most 1 in the block graph.

If $G$ contains a 3-edge-connected endblock, then the proof of Lemma~\ref{lem:main} is short. If not, then we choose a 2-edge-cut minimizing the size of a component $H$. We distinguish three cases depending on how many vertices of $S$ are contained in $H$. The proofs are short unless $H$ contains precisely one vertex of $S$, say $v$. In this case we further investigate the structure of $H-v$. Again, the proofs are short unless $H-v$ is 2-connected and $H$ contains two vertices $x$ and $y$ which have neighbours in $G-H$. If this is the case, we find an induced path $P$ from $v$ to $x$ such that $G-V(P)$ is connected and $H-V(P)$ is a tree. If $P$ has at most two edges, then we can find a $W_{a,b}$-configuration. If $P$ has at least three edges, then we consider several cycles and show that one of them satisfies (C), finishing the proof.

\begin{proof}[Proof of Lemma~\ref{lem:main}]
Let $B$ be an endblock of $G$. First suppose that $B$ is $3$-edge connected. Since $S$ contains all vertices of degree at least 4 in $G$, this implies that if there is a cut-vertex of $G$ in $B$ then that cut-vertex belongs to~$S$. If~$B$ contains at least two vertices of $S$, then we can use Lemma~\ref{lem:tutte} to find a non-separating induced path between them which then satisfies (P). Thus, we can assume that $B$ contains at most one vertex of $S$, say $v$. If $B$ contains no vertex of $S$, let $v$ denote an arbitrary vertex of~$B$. Since~$B$ is 2-connected and $B-v$ has minimum degree 2 and thus contains a cycle, we can use Lemma \ref{lem:new_carsten_toft} to find a non-separating induced cycle in $B$ not containing~$v$, and this cycle is also non-separating in $G$ and satisfies (C).

Thus, we may assume that $B$ is not 3-edge connected. If $G$ is not 2-connected let~$b$ denote the unique cut-vertex of $G$ in $B$. Choose a 2-edge cut in $B$ minimizing the size of the component $H$ not containing $b$ (if $b$ does not exist we just minimize the size of some component $H$). Note that the choice of $H$ implies that $H$ is 2-edge-connected and contained in the endblock $B$. \\

\textbf{Case 1:} $H$ contains no vertex of $S$.\\
Since every vertex in $H$ has degree 3 in $G$, every cutvertex of $H$ would give rise to a connected subgraph $H'\subset H$ which can be separated from $G-H'$ by at most 2 edges, contradicting our choice of $H$. Hence, we can assume that~$H$ is 2-connected. Let $x\in V(H)$ be a vertex joined to $G-H$. Notice that $H-x$ has minimum degree 2 and thus contains a cycle. By Lemma~\ref{lem:new_carsten_toft}, there exists a non-separating induced cycle in $H$ not containing $x$. This cycle is also non-separating in $G$ and thus satisfies (C).\\

\textbf{Case 2:} $H$ contains at least two vertices of $S$.\\
In this case, let $u_1$ and $u_2$ denote two vertices in $H$ contained in $S$. Let $P$ be an induced path from $u_1$ to $u_2$ in $H$ which maximizes the size of the connected component $K$ of $G-E(P)$ containing $G-H$.\\

\textit{Claim:} $G-E(P)$ is connected. \\ 
\emph{Proof of Claim:} Suppose $G-E(P)$ is not connected. As in the proof of Lemma \ref{lem:tutte} we have that for any vertices $z_1,z_2$ on $P$ belonging to the same component $L\neq K$ of $G-E(P)$, the $z_1z_2$-subpath of $P$ does not contain any vertices of $K$. Let $k_1$ and $k_2$ denote the first and last vertex on $P$, respectively, which is contained in $K$. If $k_1 \neq u_1$, then the last edge of the $u_1k_1$-subpath of $P$ is a cut-edge in $G$ which contradicts $H$ being 2-edge-connected.\\
Thus, we may assume $u_1=k_1$ and similarly $u_2=k_2$. Let $w$ be the first vertex on the path from $u_1$ to $u_2$ which is not contained in $K$. Let $k$ denote the first vertex on the $wu_2$-subpath of $P$ which is contained in $K$. Let $e_w$ and $e_k$ denote the last edge of the $u_1w$-subpath of $P$ and of the $u_1k$-subpath of $P$, respectively. As in the proof of Lemma \ref{lem:tutte}, the edges $e_w,e_k$ form a 2-edge-cut in $G$, contradicting the choice of $H$. Thus, every vertex of $G-E(P)$ is contained in $K$ and $G-E(P)$ is connected. $\qedhere$\\

Thus, $G-E(P)$ is connected and $P$ satisfies (P).\\

\textbf{Case 3:} $H$ contains precisely one vertex $v$ of $S$.\\
We distinguish three cases depending on the structure of $H-v$. Notice that if~$H$ contains a cutvertex $w$, then by 2-edge-connectivity of $H$, the vertex $w$ has at least two neighbours in every block of $H$ it is contained in. In particular, $w$ has at least degree 4 and thus $w=v$ is the only possible cutvertex of $H$. Thus, Case~3.1 is identical to the case where $H$ is not 2-connected.\\

\textbf{Case 3.1:} $H-v$ is disconnected.\\
Let $B_H$ be a block of $H$ which contains at most one vertex with a neighbour in $G-H$. Now $B_H-v$ contains at most one vertex of degree 1, hence there exists a cycle in $B_H-v$. By Lemma~\ref{lem:new_carsten_toft}, $B_H$ contains a non-separating induced cycle $C$ avoiding $v$, see Figure~\ref{fig:lem_3_6_case_31}. The cycle $C$ is also non-separating in $G$ and thus satisfies (C).\\

\begin{figure}
    \centering
    \begin{minipage}{0.4\textwidth}
    \includegraphics[scale=.5]{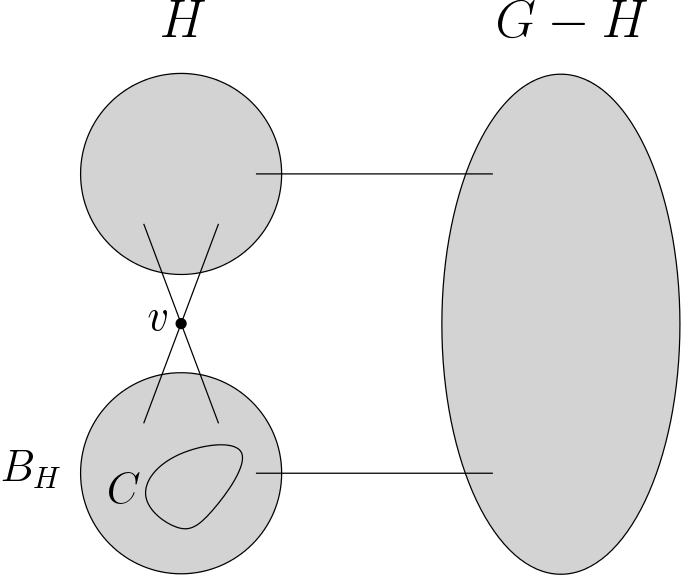}
    \caption{$H-v$ is disconnected}
    \label{fig:lem_3_6_case_31}
    \end{minipage}
    \hfill
    \begin{minipage}{0.4\textwidth}
     \includegraphics[scale=.5]{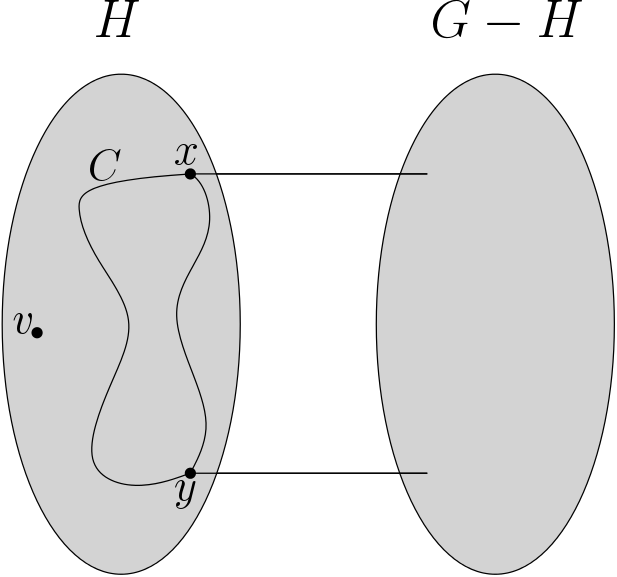}
    \caption{Case 3.3}
    \label{fig:lem_3_6_case_33}
    \end{minipage}
\end{figure}

\textbf{Case 3.2:} $H-v$ is a tree.\\
Notice that every vertex of degree 1 in $H-v$ must have a neighbour in $G-H$. Hence, there are at most two vertices of degree 1 in $H-v$. In particular, $H-v$ is a path $P=xv_1\cdots v_ay$ where both $x$ and $y$ have a neighbour in $G-H$. Thus, there exists a $W_a$-configuration in $G$ and (W) is satisfied. \\

\textbf{Case 3.3:} $H-v$ is connected and contains a cycle.\\
By Lemma~\ref{lem:new_carsten_toft}, there exists a non-separating induced cycle $C$ in $H$ avoiding $v$. If $C$ does not satisfy (C), then $C$ contains every vertex of $H$ with neighbours in $G-H$. If there exists only one such vertex $x$, then $x$ has degree 3 and $H'=H-x$ would be a smaller graph that can be separated from $G-H'$ by a 2-edge-cut, contradicting the minimality of $H'$. Thus, we can assume that $C$ contains two vertices $x$ and $y$ which have neighbours in $G-H$, see Figure~\ref{fig:lem_3_6_case_33}. Notice that $x$ and $y$ are not adjacent since otherwise the graph $H-x-y$ would contradict the minimality of $H$.\\

\textbf{Case 3.3.1:} $H-v$ has a cut-vertex $w$.\\
Since $w$ has degree at most 3 in $H-v$, there exists an edge $e$ incident with $w$ such that $H'=H-v-e$ is disconnected. Notice that since $C$ does not contain $v$, and since $e$ is a cut-edge in $H-v$, the cycle $C$ also exists in $H'$. In particular, $x$ and $y$ are contained in the same block of $H-e$. By definition, $v$ is a cutvertex in $H-e$. Suppose $H-e$ contains a cutvertex $v'\neq v$. Since $v'$ has degree at most 3 in $H-e$, there exists an edge $e'$ which is a cut-edge in $H-e$. Now $e,e'$ form a 2-edge-cut which contradicts our choice of $H$, see Figure~\ref{fig:lem_3_6_case_331_1}. Thus, $v$ is the only cutvertex in $H-e$ and $H-e$ has exactly two blocks. Let $B_H$ be the block of $H-e$ containing neither $x$ nor $y$. Notice that $B_H$ contains an end of $e$. There exists at most one vertex of degree 1 in $B_H-v$, hence $B_H-v$ contains a cycle. By Lemma~\ref{lem:new_carsten_toft}, there exists a non-separating induced cycle $C'$ in $B_H$ avoiding $v$. The cycle $C'$ is also non-separating in $G$ and thus satisfies~(C).\\
\begin{figure}
    \centering
    \includegraphics[scale = 0.8]{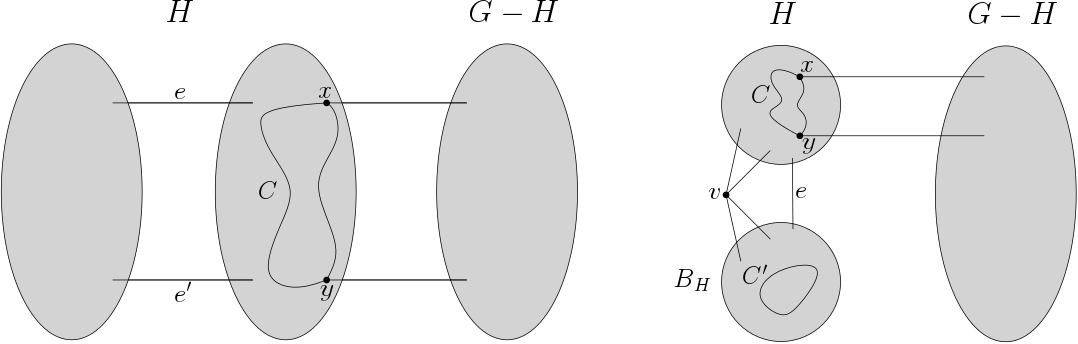}
    \caption{Case 3.3.1}
    \label{fig:lem_3_6_case_331_1}
\end{figure}

\textbf{Case 3.3.2:} $H-v$ is 2-connected. \\

\textit{Claim:} There exists an induced path $P$ in $H$ from $v$ to $x$ such that $G-V(P)$ is connected. \\ 
\emph{Proof of Claim:} Let $P$ be an induced path in $H$ from $v$ to $x$, such that the size of the component $K$ of $G-V(P)$ containing $G-H$ is maximum. Notice that since $v$ is the only vertex of degree greater than 3 in $H$, the graph $G-V(P)$ is connected if and only if $G'=G-E(P)$ is connected and $v$ is not a cutvertex in $G'$. First, suppose that $G'$ is disconnected and let $K'$ be the component of $G'$ containing $G-H$. As in the proof of the claim in Case 2 and the proof of Lemma \ref{lem:tutte}, we have that for any vertices $z_1,z_2$ on $P$ belonging to the same component $L\neq K'$ of $G'$, the $z_1z_2$-subpath of $P$ does not contain any vertices of $K'$. Note that $x \in V(K')$ since $x$ has a neighbour in $G-H$ and as in the proof of Lemma \ref{lem:tutte} we can also assume $v \in V(K')$. Let $w$ be the first vertex on the path from $v$ to $x$ which is not contained in $K'$. Let $k$ denote the first vertex on the $wx$-subpath of $P$ which is contained in $K'$. As in the proof of the claim in Case 2 we find two edges incident to $w$ and $k$, respectively, which form a 2-edge-cut in $G$ contradicting the choice of $H$. Thus $G'$ is connected.\\ 
Now suppose that $v$ is a cutvertex in $G'$. Let $z$ denote the first vertex on $P$ after $v$ which has a neighbour in $K$. Let $L$ denote a component different from $K$ in $G-V(P)$, in particular $L$ is adjacent to $v$. Similar to before, no component $M\neq K$ of $G-V(P)$ has two neighbours $z_1$, $z_2$ on $P$ such that $z$ is contained in the $z_1z_2$-subpath of $P$. Thus, the graph $H-\{v,z\}$ is disconnected, which contradicts 2-connectivity of $H-v$. \hfill $\square$ \\

Let $P$ be the path from the claim above. Clearly we must have $y \notin V(P)$. If $H-V(P)$ contains a cycle, then we can use Lemma~\ref{lem:new_carsten_toft} to find an induced cycle $C'$ avoiding $P$ for which $H-E(C')$ is connected. Since $C'$ does not contain $x$, it is also non-separating in $G$ and thus satisfies (C). Therefore, we may assume that $T=H-V(P)$ is a tree. We distinguish two cases depending on the length of $P$.\\

\textbf{Case 3.3.2.1} $|E(P)|\leq 2$.\\
Since $x$ is contained in a cycle not containing~$v$, we have $|E(P)|=2$. Let $u$ denote the middle vertex of $P$. Since $y$ is not adjacent to $v$, each leaf of $T$ is adjacent to at least one of $x$ and $u$. The vertices $x$ and $u$ are each only adjacent to one vertex outside of $P$, thus the tree $T$ can contain at most two leaves and is therefore a path. Every vertex on $T$ of degree 2 and different from $y$ is adjacent to $v$. This shows that there exists a $W_{a,b}$-configuration (with $v$ as its centre) in $G$ and (W) is satisfied.  \\

\begin{figure}
    \centering
    \includegraphics[scale = 1.5]{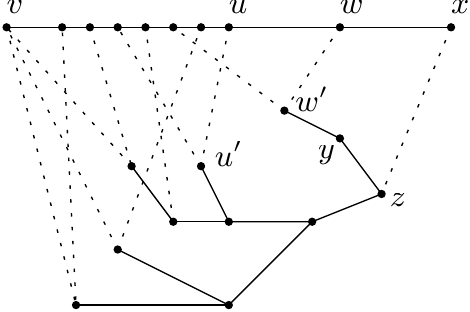}
    \caption{An example of how $T$ in Case 3.3.2.2 could look like.}
    \label{fig:case3322}
\end{figure}
\textbf{Case 3.3.2.2} $|E(P)|\geq 3$.\\
Let $u$ denote the vertex at distance 2 of $x$ on $P$, and $w$ the neighbour of $x$ on $P$. Since $P$ is induced and $u$ and $w$ have degree 3 in $H$, there exist vertices $u', w'$ adjacent to $u, w$ and not contained in $P$. First suppose $u' = w'$. Since $u'$ has degree 3 and $T$ is connected, the vertex $u'$ is not adjacent to $x$. Now $uu'w$ is a non-separating induced cycle in $H$ which also satisfies (C). Thus, we may assume that $u'\neq w'$. Let $P'$ denote the (unique) path in $T$ connecting $u'$ and $w'$. Let $C_P$ denote the induced cycle consisting of $P'$ and the edges $w'w, wu,uu'$. Notice that every component of $T-V(P')$ contains at least one leaf of $T$ and that every such leaf apart from $y$ is adjacent to two vertices of $P-u-w$. Since $x$ is not adjacent to $y$, it follows that every component of $T-V(P')$ has a neighbour in $V(P)\setminus\{u,w,x\}$. Thus, $H-x-V(C_P)$ is connected. If $y$ is not contained in $C_P$ or $x$ has a neighbour in $H-V(C_P)$, then $C_P$ satisfies (C). See Figure~\ref{fig:case3322} for a specific example of $H$ where $C_p$ does not satisfy (C).

If $C_P$ does not satisfy (C), let $z$ denote the neighbour of $x$ in $T$ and define the cycles $C_w$ and $C_u$ as follows: the cycle $C_w$ consists of the $zw'$-subpath of $P'$ together with the edges $w'w, wx, xz$, while $C_u$ consists of the $zu'$-subpath of $P'$ together with the edges $u'u,uw,wx,xz$, see Figure~\ref{fig:lem3_6_case_3322_cycles.png}. It is easy to see that $C_w$ satisfies (C) unless $y$ is contained in the $zw'$-subpath of $P'$. Thus, we may assume that $y$ is contained in the $zw'$-subpath and $y\neq z$. Notice that this implies $w'\neq z$ and hence $C_u$ is induced. Suppose that $G-V(C_u)$ is not connected. Clearly $G-V(C_u)$ has at most two components: one containing $w'$ and one containing $v$. Let $K_{w'}$ denote the connected component of $G-V(C_u)$ containing $w'$. Let $\ell$ be a leaf in $T$ contained in $K_{w'}$. If $\ell \neq y$, then $\ell$ has a neighbour on $P-V(C_u)$ and therefore $K_{w'} = G-V(C_u)$. Thus, we may assume that $y$ is a leaf in $T$ and no other leaf of $T$ is contained in $K_{w'}$. Since $y$ is contained in $C_w$ we have $w'=y$. If $y$ is not adjacent to $z$, then there exists a vertex $k$ distinct from $z,w'$ on the $zw'$-path in $T$. Since $K_{w'}$ contains only one leaf of $T$, the vertex~$k$ has degree 2 in $T$. Therefore, $k$ has a neighbour on $P-V(C_u)$ and we again get the contradiction $K_{w'} = G-V(C_u)$. Finally, suppose that $y$ is adjacent to $z$, see Figure~\ref{fig:lem3_6_case_3322_new_2_cut.png}. In particular, $y$ is the only vertex in $K_{w'}$. Now the graph $H'=H-x-y-w-z$ can be separated from $G-H'$ by a 2-edge-cut, contradicting our choice of $H$.\hfill $\square$
\end{proof}

\begin{figure}
    \centering
    \includegraphics[scale=.7]{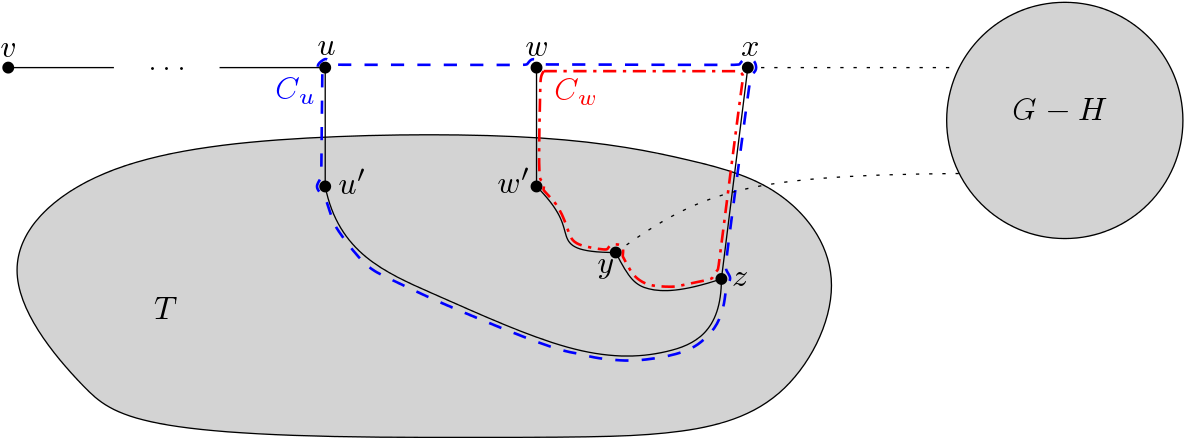}
    \caption{The cycles $C_u$ and $C_w$ in Case 3.3.2.2.}
    \label{fig:lem3_6_case_3322_cycles.png}
\end{figure}
\begin{figure}
    \centering
    \includegraphics[scale=.75]{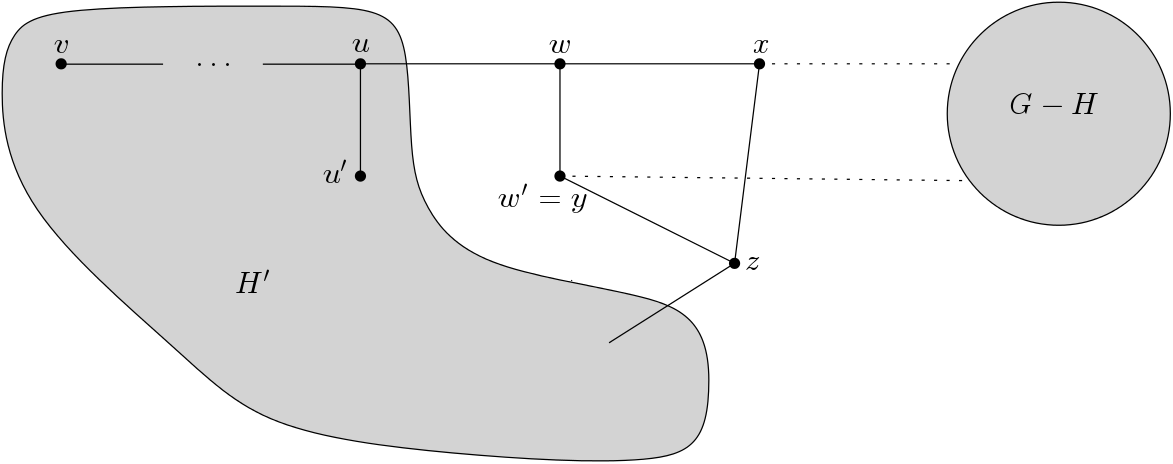}
    \caption{Case 3.3.2.2 where $y=w'$ and $y$ is adjacent to $z$.}
    \label{fig:lem3_6_case_3322_new_2_cut.png}
\end{figure}

\bibliographystyle{acm}
\bibliography{bibliography}
\end{document}